\documentclass[12pt,a4paper]{amsart}
\usepackage{amssymb,latexsym}

\usepackage{graphicx,amssymb,amsfonts,epsfig,amsthm,a4,amsmath,url}

\newtheorem{Atheorem}{Theorem}[section]
 
\newtheorem{Acorollary}[Atheorem]{Corollary}

\newtheorem{Aproposition}[Atheorem]{Proposition}

\newtheorem{thm}{Theorem}[subsection]

\newtheorem{corollary}[thm]{Corollary}
\newtheorem{theorem}[thm]{Theorem}

\newtheorem{lemma}[thm]{Lemma}

\newtheorem{proposition}[thm]{Proposition}

\theoremstyle{definition}

\newtheorem{definition}[thm]{Definition}

\theoremstyle{remark}

\numberwithin{equation}{section}

\newcommand{\Z}{\mathbf{Z}}

\newcommand{\N}{\mathbf{N}}
\newcommand{\R}{\mathbf{R}}

\newcommand{\Q}{\mathbf{Q}}

\newcommand{\CB}{\textsc{cb}}

\newcommand{\Cp}{C_{p^{\infty}}}
\newcommand{\SA}{\mathcal{S}(A)}
\newcommand{\Cond}{\textsc{Cond}}
\newcommand{\ovCB}{\overline{\textsc{cb}}}
\newcommand{\Div}{\textnormal{Div}}
\newcommand{\Aut}{\textnormal{Aut}}
\newcommand{\GL}{\textnormal{GL}}
\newcommand{\Ker}{\text{Ker}}

\newcommand{\Hom}{\textnormal{Hom}}

\newcommand{\Crit}{\textnormal{cr}}
\newcommand{\Ncrit}{\textnormal{ncr}}

\newcommand{\Ord}{\textsc{Ord}}

\begin{document}

\title{The space of subgroups of an abelian group}

\author{Yves de Cornulier}
\address{IRMAR \\ Campus de Beaulieu \\
35042 Rennes Cedex, France}
\email{yves.decornulier@univ-rennes1.fr}

\author{Luc Guyot}

\address{Institut Fourier \\
Universit\'e de grenoble \\
B.P. 74 \\ 38402 Saint Martin d'H\`eres Cedex \\ France}
\email{Luc.Guyot@ujf-grenoble.fr}

\author{Wolfgang Pitsch}
\address{Universitat Aut\`onoma de Barcelona \\ Departament de Matem\`atiques\\
E-08193 Bellaterra, Spain}
\email{pitsch@mat.uab.es}

\thanks{The third author is supported  by MEC grant MTM2004-06686 and  by the program Ram\'on y Cajal, MEC, Spain}



\date{July 14, 2010}
\maketitle

\begin{abstract}
We carry out the Cantor-Ben\-dix\-son analysis of the space of all
subgroups of any countable abelian group and we deduce a complete
classification of such spaces up to homeomorphism.
\end{abstract}

\section*{Introduction}
Let $G$ denote a discrete group. The set $\mathcal{N}(G)$ of all normal subgroups of $G$ has a natural topology, called the Chabauty topology.
It is the
setting of an interesting interplay between topological phenomena and
algebraic properties of the group $G$. Introduced by Chabauty in \cite{chb}, it reappeared
in the work of Gromov \cite{Grom81} and Grigorchuk \cite{Gri84}, where it proved to be a useful tool to understand asymptotic properties
of discrete groups, see for instance \cite{ChGu05,ABL+05}.
More precisely, consider the set of subsets of $G$, viewed as
the product $2^G$, endowed with the product (Tychonov) topology.
The subset $\mathcal{N}(G)$ is easily seen
to be closed. By construction, this is a compact, totally disconnected Hausdorff topological space, that is, a Boolean space. If $G$ is countable, then it is metrizable. For this topology, a net $(N_i)$ of normal subgroups converges to $N$ if and only if for every $g\in N$ (resp. $g\notin N$), eventually $g\in N_i$ (resp. $g\notin N_i$).

Very little is known on the global structure of $\mathcal{N}(G)$ of a group $G$, however see \cite{Cham00,CGP07}, which especially deal with the case when $G$ is free of finite rank. In this article we treat the
case where the group $G$ is abelian (not necessarily finitely generated) and we preferably write $A$ in
place of $G$. We obviously have $\mathcal{N}(A)=\SA$, the space of subgroups of $A$. This is not unrelated to finitely generated groups: indeed, if $A$ is any countable abelian group, then $A$ embeds into the center of a finitely generated group $G$, giving an obvious embedding of $\SA$ into $\mathcal{N}(G)$.
The
classification of the spaces $\mathcal{S}(A)$ for an abelian group
$A$ turns out to be much more tractable than both of its natural and
difficult generalizations, namely the classification of either
general Boolean spaces or general abelian groups, see for instance
\cite{Pie70} and \cite{Tho01} concerning these problems. Indeed, we
are able to provide a complete description of the spaces
$\mathcal{S}(A)$ in terms of natural, and computable, invariants of
the countable abelian group $A$.

A topological space $X$ is called {\it perfect} if it has no isolated point, and, at the other extreme, {\it scattered} if any non-empty subset has an isolated point. As an union of perfect subsets is perfect, every topological space has a unique largest perfect subset, called its {\it condensation part} and denoted $\Cond(X)$, which is empty if and only if $X$ is scattered. The subset $X-\Cond(X)$ is the largest scattered open subset, and is called the {\it scattered part} of $X$.

If $A$ is an abelian group, its torsion elements form a subgroup denoted by $T_A$. Recall that an element of $A$ is called divisible if it belongs to $nA$ for all non-zero integers $n$. The set of divisible elements in $A$ form a subgroup denoted by $\Div(A)$ and it is easy to check that it always has a direct complement in $A$.
Given a prime $p$, we define $\Cp=\Z[1/p]/\Z$; this is called a quasi-cyclic group. An abelian group is called \textit{Artinian} if every non-increasing sequence of subgroups stabilizes; every such group is isomorphic to a direct sum $A=\bigoplus_{i=1}^hC_{p_i^\infty}\oplus F$ for some finite subgroup $F$; the finite index subgroup $\bigoplus_{i=1}^hC_{p_i^\infty}$ coincides with $\Div(A)$.
An abelian group $A$ is called \textit{minimax} if it has a finitely generated subgroup $Z$ with $A/Z$ Artinian. Such a subgroup $Z$ is called a \textit{lattice} in $A$. 


\begin{Aproposition}[Corollaries \ref{cor:existisolated} and \ref{cor:cantor}] \label{PropNonNA}
Let $A$ be an abelian group. Then $\SA$ is non-perfect (i.e. contains an isolated point) if and
only if $A$ is minimax. In particular,
\begin{itemize}\item if $A$ is countable and not                 
minimax, then $\SA$ is homeomorphic to a              
Cantor set;
\item if $A$ is uncountable, then $\SA$ is perfect.\end{itemize}     
\end{Aproposition}

When the abelian group $A$ is uncountable, we do not have\footnote{Added after publication: A result due to Tsybenko \cite{Tsy86} contains and improves Proposition~A, showing in addition that if the countable    
discrete abelian group $G$ has cardinality $\aleph_1$, then $\mathcal{S}(G)$     
is homeomorphic to $2^{\aleph_1}$, and conversely if $G$ has cardinality        
$>\aleph_1$, then $\mathcal{S}(G)$ is {\em not} dyadic, that is, it is not a    
continuous image of a Cantor cube $2^\tau$ for any cardinal $\tau$.   } any classification result for the perfect space $\SA$, except the following, which shows that the cardinality of $A$ can be read out of the topology of $\SA$.

\begin{Aproposition}[see Paragraph \ref{subs:uncount}]
Suppose that the abelian group $A$ has uncountable cardinal $\alpha$. 
Then $\SA$ contains a subset homeomorphic to $2^\alpha$. Accordingly, $\alpha$ is the least cardinality of a basis for the topology of $\SA$.
In particular, if $A$ is uncountable, then $\SA$ is not metrizable.
\label{uncountable}
\end{Aproposition}

Our main result is the determination of the homeomorphism type of $\SA$ when $A$ is a countable abelian group. Proposition \ref{PropNonNA} settles the case of non-minimax ones.

Recall that the \textit{rank} of an abelian group $A$, denoted $r(A)$, is the largest cardinal of $\Z$-free family in $A$; this is also the $\Q$-dimension of the vector space $A\otimes_\Z\Q$. If $A$ is minimax, then $r(A)<\infty$.
We also have to introduce the notion of critical prime of a minimax abelian group, which plays a crucial role here.
If $A$ is a minimax abelian group and $p$ a prime, we define $\ell_p(A)$ as the largest integer $k$ so that $A$ maps onto $\Cp^k$. Given a lattice $Z$ in $A$, this is also the greatest integer $k$ so that $\Cp^k$ embeds into $A/Z$. The sum $$h(A)=r(A)+\sum_{p\text{ prime}}\ell_p(A)$$ is finite and called the \textit{height} of $A$. 
A prime $p$ is called \textit{critical} for $A$ if $\ell_p(A)\ge 2$. The set of critical, respectively non-critical, primes for $A$ is denoted by $\Crit(A)$, resp. $\Ncrit(A)$. The minimax abelian group $A$ is called \textit{critical} if $\Crit(A)\neq\emptyset$.
Finally, if $A$ is a minimax abelian group, $\Div(A)$ is contained in $T_A$ as a subgroup of finite index; the number of subgroups of the finite group $T_A/\Div(A)$ is denoted by $n(A)$.


Let $[n]$ denote the set $\{0,\dots,n-1\}$ with $n$ elements, and $\omega=\bigcup_n[n]$. Let $D$ denote the topological space $\omega\cup\{\infty\}$ consisting of the discrete sequence of points $(n)_{n\ge 0}$ converging to the limit $\infty$, which is homeomorphic to the subset $\{1/n|\,n\ge 1\}\cup\{0\}$ of $\R$. For any integer $m$, the space $D^m$ is scattered, since $D$ is scattered and the class of scattered spaces is closed under finite cartesian products.

\begin{Atheorem}[Theorem \ref{thm:nonscat}] \label{ThScat}
Let $A$ be a non-critical minimax abelian group, and write $h=h(A)$, $n=n(A)$. Then the space $\mathcal{S}(A)$ is countable, and homeomorphic to $D^{h}\times [n]$.
\end{Atheorem}

All the cases occur (for $h\ge 0$ and $n\ge 1$), for instance the abelian group $A=\Z^h\oplus\Z/2^{n-1}\Z$ has $h(A)=h$ and $n(A)=n$, and as a finitely generated group, it is minimax and non-critical.

Let us now deal with critical minimax abelian groups. Again, we have to introduce some more definitions.
Let $A$ be a minimax abelian group and let $V$ be a set of primes. First define $$\ell_V(A)=\sum_{p\in V}\ell_p(A).$$
Let $R_V\subset\Q$ denote the ring of rationals whose denominator is divisible by no $p\in V$. Let $U_V(A)\le A$ be the intersection of kernels of homomorphisms $A\to R_V$. Set $a_V(A)=r(U_V(A))$ and $\gamma_{V^c}(A)=r(A/U_V(A))$. Note that $A/U_V(A)$ embeds into $R_V^{\gamma_{V^c}(A)}$.

On the other hand, let $W$ be the ``dusty Cantor space"; namely a compact metrizable space consisting of the union of a Cantor space with an open dense countable discrete set. It is a consequence of Pierce's theorem (Theorem \ref{thpierce}) that $W$ is thus uniquely defined, up to homeomorphism.
For instance, $W$ is chosen as the union of the triadic Cantor set (which is the image of $\{0,2\}^\N$ in $\R$ by the injective continuous map $u\mapsto\sum_{n\ge 0} u(n)3^{-n-1}$) and the set of centers of all intervals in the complement, namely the set of reals of the form $\sum_{n\ge 0}v(n)3^{-n-1}$, where $v(n)$ is a sequence such that for some $n_0\ge 0$, $v(n)\in\{0,2\}$ for all $n<n_0$ and $v(n)=1$ for all $n\ge n_0$. 

\begin{Atheorem}[Theorem \ref{thm:nonscat}]\label{thm:nonscatA}
Let $A$ be a critical minimax abelian group. Then $\SA$ is uncountable, and homeomorphic to $D^\sigma\times W$, where $\sigma=\sigma(A)$ is defined as follows
\begin{eqnarray*}
\sigma(A) & = & \gamma_{\Ncrit(A)}(A)+\ell_{\Ncrit(A)}(A) \\
& = &  h(A)-(a_{\Crit(A)}(A)+\ell_{\Crit(A)}(A)).
\end{eqnarray*}
\end{Atheorem}

Again, all the cases occur: the minimax group ${(C_{2^\infty})}^2\times\Z^\sigma$ is critical, has $\sigma(A)=\sigma$. Even better, for a minimax group $A$ with given $h(A)$, if $\SA$ is uncountable, then $0\le \sigma(A)\le h(A)-2$, and all these cases occur, taking $A={(C_{2^\infty})}^{h-\sigma}\times\Z^\sigma$, we have $h(A)=h$ and $\sigma(A)=\sigma$ provided $0\le\sigma\le h-2$.

From the conjonction of Theorems \ref{ThScat} and \ref{thm:nonscatA}, we get the following corollaries.

\begin{Acorollary}[Boyer \cite{Boy56}]\label{sacount}
Let $A$ be an abelian group. Then $\SA$ is countable if and only if $A$ is a non-critical minimax group.
\end{Acorollary}

Note that the harder implication, namely the forward one, follows from Theorem \ref{ThScat}.

\begin{Acorollary}[see Lemma \ref{nonhomeo}] \label{CorClass}
Let $A,B$ be countable abelian groups. The spaces $\mathcal{S}(A)$
and $\mathcal{S}(B)$ are homeomorphic if and only if one of the
following holds:
\begin{itemize}
\item[$(i)$] $h(A)=h(B)=\infty$;
\item[$(ii)$] both $A$ and
$B$ are minimax and non-critical, $h(A)=h(B)$, and $n(A)=n(B)$;
\item[$(iii)$] both $A$ and $B$ are minimax and critical, and $\sigma(A)=\sigma(B)$.
\end{itemize}
\end{Acorollary}

Next, we answer the question ``where" does a given subgroup $S$ of $A$ lie in $\SA$.

Let $X$ be a topological space. Let $X^{(1)}$ be its derived subspace, that is the set of its accumulation points, i.e. non-isolated points. Define by induction $X=X^{(0)}$ and $X^{(n+1)}=(X^{(n)})^{(1)}$. If $x\in X^{(n)}-X^{(n+1)}$, we say that the Cantor-Bendixson rank\footnote{We are avoiding the use of ordinals in this introduction. See Paragraph \ref{subs:cba} for the extension of the Cantor-Bendixson as an ordinal-valued function defined on all the scattered part of $X$.} of $x$ in $X$ is $n$ and we write it $\CB_X(x)=n$. Note that $\Cond(X)\subset X^{(n)}$ for all $n$.
If $X$ is a metrizable Boolean space, it can be checked that the Cantor-Bendixson rank of $x$ in $X$ is the integer $n$ if and only if there exists a homeomorphism of a neighbourhood of $x$ (which can be chosen clopen) to $D^n$ mapping $x$ to $(\infty,\dots,\infty)$.
If there is some $n$ such that $\CB_X(x) \le n$ whenever $\CB_X(x)$ is defined \footnote{This is the case if $X=\SA$ with $A$ an arbitrary abelian group, as a consequence of Proposition \ref{PropNonNA} and Theorems \ref{ThScat} and \ref{thm:nonscatA}.},  we also define for any $x\in X$, its \textit{extended Cantor-Bendixson rank} as
$$\ovCB_X(x)=\inf\sup\{\CB_X(y)|\,y\in V,\CB_X(y)<\infty\},$$
where the infimum ranges over all neighbourhoods of $x$. This extends the function $\CB_X$ (assuming $\sup\emptyset=-\infty$).

For a minimax abelian group $A$ and a prime $p$, define $$\kappa_p(A)=\ell_p(A/T_A) \text{ and } \tau_V(A)=\ell_p(T_A);$$
and, for a set of primes $V$,
$$\kappa_V(A)=\sum_p\kappa_p(A) \text{ and } \tau_V(A)=\sum_p\tau_p(A).$$
If $S\le A$, define
\begin{eqnarray*}
d_A(S)& = &  \gamma_{\Ncrit(A)}(A/S)+\kappa_{\Ncrit(A)}(A/S)+\ell_{\Ncrit(A)}(S) \\
 & = & \sigma(A)- (\tau_{\Ncrit}(A/S)+\gamma_{\Ncrit(A)}(A)-\gamma_{\Ncrit(A)}(A/S))
\end{eqnarray*}

\begin{Atheorem}[Theorem \ref{thm:nonscat}] \label{ThFormula}
Let $A$ be a minimax abelian group and let $S \le
A$. 
We have
\begin{itemize}
\item[$(i)$] $S$ is in the scattered part of $\SA$ if and only if $$\kappa_{\Crit(A)}(A/S)=\ell_{\Crit(A)}(S)=0;$$
\item[$(ii)$] the extended Cantor-Bendixson rank of $S$ in $\SA$ is $d_A(S)$.
\end{itemize}
\end{Atheorem}

\setcounter{tocdepth}{1}
\tableofcontents

The outline of the article is as follows. In Section
\ref{sec:toppre}, we establish some topological preliminaries, notably a characterization of the extended Cantor-Bendixson function by semi-continuity and hereditary properties (Paragraph \ref{sch}) and a topological characterization of the spaces involved in Theorems \ref{ThScat} and 
\ref{thm:nonscatA} (Paragraph \ref{cha}). The short Section \ref{gen} is devoted to general abelian groups, and we prove Propositions \ref{PropNonNA} and \ref{uncountable} there. Sections \ref{wei} and \ref{lwei} are devoted to the study of $\SA$ when $A$ is a a minimax group. Section \ref{wei} contains enough to prove Theorem \ref{ThScat}, while both sections are necessary to obtain Theorems \ref{thm:nonscatA} and \ref{ThFormula}.

Sections \ref{wei} and \ref{lwei} contain a number of preliminary results, pertaining to the topology of commensurability classes in $\SA$ for an abelian minimax group $A$, which can be of independent interest.



\section{Topological preliminaries}\label{sec:toppre}

\subsection{Cantor-Bendixson analysis}\label{subs:cba}

Let $X=X^{(0)}$ be a topological space. If $X^{(1)}$ is its derived subspace, one can define by transfinite induction $X^{(\alpha)}$ as the derived subspace of $\bigcap_{\beta<\alpha}X^{(\beta)}$. This is a decreasing family of closed subsets of $X$, and the first $\alpha$ such that $X^{(\alpha)}$ is perfect is called the Cantor-Bendixson rank of $X$ and denoted $\CB(X)$. The advantage of this ordinal-valued definition is that $\Cond(X)=\bigcap_\alpha X^{(\alpha)}$ (if we restrict to integers, this is only an inclusion $\subset$ in general). If $x\notin\Cond(X)$, its Cantor-Bendixson rank is defined as $$\CB_X(x)=\sup\{\alpha|\,x\in X^{(\alpha)}\}.$$ This function is extended to all of $X$ by
$$\ovCB_X(x)=
\inf\sup\{\CB_X(y)|\,y\in V-\Cond(X),\},$$
where the infimum ranges over all neighbourhoods of $x$, assuming $\sup\emptyset=-\infty$.

\subsection{Semi-continuity, heredity}\label{sch}
\begin{definition}\label{def lwsemcont}
Let $X$ be a topological space. A map $f: X \rightarrow\Ord$ is called \emph{upper semi-continuous} at $x_0\in X$ if $f$ has a local maximum at $x_0$. It is
\emph{strictly} upper semi-continuous at $x_0$ if it is a strict local maximum, i.e. if $f(x)<f(x_0)$ when $x$ is close enough to $x_0$.
\end{definition}

\begin{definition}
Let $X$ be a topological space.
\begin{itemize}\item If $Y\subset X$ is a dense subset, a map $X\to\Ord$ is $Y$-\textit{hereditary} if for every $x\in X$ and for every neighbourhood $V$ of $x$ in $X$, we have $$f(x)\le\sup_{y\in V\cap Y}f(y)$$
\item A map $X\to\Ord$ is \textit{strictly hereditary} if for every $x\in X$ and for every neighbourhood $V$ of $x$ in $X$, we have $$f(x)\le\sup_{x'\in V\setminus\{x\}}(f(x')+1),$$
where we set $\sup\emptyset=0$.
\end{itemize}\end{definition}

\begin{lemma}\label{carcb}
Let $X$ be a topological space, with two subsets $I,C \subset X$ and a map $f:\overline{I}
\longrightarrow \Ord$. Assume that the following conditions
are satisfied:
\begin{itemize}
\item[(i)]$C$ has no isolated point;
\item[(ii)]$X=I\cup C$;
\item[(iii)]$f$ is upper semi-continuous and $I$-hereditary on $\overline{I}$;
\item[(iv)]$f$ is strictly upper semi-continuous and strictly hereditary on $I$.
\end{itemize}
Then $I=\mathcal{I}(X)$, $C=\Cond(X)$, and $f$ coincides with $\ovCB_X$ on $\overline{I}$.
\end{lemma}
\begin{proof}

First, using (iv), by induction on $\alpha$, if $x\in I$ and $f(x)=\alpha$, we get $x\in\mathcal{I}(X)$ and the Cantor-Bendixson rank of $x$ in $X$ is $\alpha$. Hence $I \subset \mathcal{I}(X)$.

By (i), $C\subset\Cond(X)$.
Combining and using (ii), we get $I=\mathcal{I}(X)$ and $C=\Cond(X)$.
Finally using (iii), we get that $f$ coincides with the extended Cantor-Bendixson rank on $C\cap\overline{I}$. 
\end{proof}

\subsection{Characterization of some topological spaces}\label{cha}

It is very useful to have a characterization of some topological spaces. For instance, we already used in the introduction the classical fact (see \cite[Theorem 7.4]{Kec95}) that if a non-empty topological space is metrizable, compact, perfect and totally disconnected, then it is homeomorphic to the Cantor set.

The second case concerns scattered spaces. It is known \cite{MS20} that a non-empty Hausdorff compact scattered topological space is characterized, up to homeomorphism, by its Cantor-Bendixson rank (an arbitrary ordinal, countable in the metrizable case), and the number of points of maximal Cantor-Bendixson rank (an arbitrary positive integer). For our purposes it is enough to retain

\begin{proposition}\label{mazur}
Let $X$ be a Hausdorff compact scattered topological space of finite Cantor-Bendixson rank $m+1\ge 1$, with $n\ge 1$ points of maximal Cantor-Bendixson rank ($=m$). Then $X$ is homeomorphic to $D^m\times[n]$.
\end{proposition}

This applies for instance to the ordinal $\omega^m\cdot n+1$.

If $X$ is a topological space, let us define $$C_{\alpha}(X):=\{
x \in \overline{\mathcal{I}(X)} \cap \Cond(X)\, \vert \, \ovCB_X(x) \ge
\alpha\}.$$

\begin{proposition} \label{PropHomeo}
Let $X$ be a metrizable Boolean space and $\sigma<\infty$.
Assume that $\CB(X)<\infty$ and
\begin{itemize}
\item[$(i)$] $\CB(X)=\sigma+1$ and $C_i(X)$ is homeomorphic to the Cantor space for all
$i\le \sigma$;
\item[$(ii)$] $C_{i+1}(X)$ has empty interior in $C_i(X)$
for all $i$.
\end{itemize}
Then $X$ is
homeomorphic to $D^\sigma\times W$.
\end{proposition}

We make use of the following result of Pierce \cite[Th.1.1]{Pie70}. Here, we assume that a point not in the closure the scattered part has extended Cantor-Bendixson rank $-\infty$.

\begin{theorem}[Pierce]
\label{thpierce}
Let $X,Y$ be metrizable Boolean spaces. Let $\phi$ be a homeomorphism $\Cond(X)\to \Cond(Y)$. Suppose that $\mathcal{I}(X)$ is homeomorphic to $\mathcal{I}(Y)$ and that $\phi$ preserves the extended Cantor-Bendixson rank. Then $\phi$ extends to a homeomorphism $X\to Y$.
\end{theorem}

We also need the following lemma.
Let us view $D^{m-1}$ as the subspace $D^{m-1}\times\{\infty\}$ of $D^m$ ($D^0$ being a singleton).

\begin{lemma} \label{LemOmegaMuK}
Let $K:=\{0,1\}^{\N}$ be the Cantor discontinuum and let $0 \le
\sigma< \omega$. Let $K=K_\sigma \supset K_{\sigma-1} \supset \dots \supset
K_0$ be subsets of $K$ all homeomorphic to $K$ and such that
$K_i$ has empty interior in $K_{i+1}$ for each $0 \le i \le \sigma-1$.
Then there is a homeomorphism $K\to K\times D^\sigma$ mapping $K_i$ to $K\times D^{i}$ for each $0\le i\le\sigma$.
\end{lemma}

The proof is an induction based on the following theorem.

\begin{theorem}\cite[Th.2]{KR53} \label{ThHomeoExt}
Let $K_i$ ($i=1,2$) be topological spaces homeomorphic to the Cantor set and let $C_i \subset K_i$ be closed subsets with empty interior in $K_i$. Assume that there is a homeomorphism $h:C_1 \longrightarrow C_2$. Then there is a homeomorphism $\widetilde{h}:K_1 \longrightarrow K_2$ extending $h$.
\end{theorem}

\begin{proof}[Proof of Lemma \ref{LemOmegaMuK}]
The proof is an induction on $\sigma$.

Step 1: The result is obvious if $\sigma=0$ and follows from Theorem \ref{ThHomeoExt} if $\sigma=1$.

Step 2: Assume now $\sigma>1$. Apply Step 1 to get a homeomorphism $\phi:K\to K\times D$ with $\phi(K_{\sigma-1})=K\times\{\infty\}$. By induction, there exists a homeomorphism $\psi:K\to K\times D^{\sigma-1}$ mapping $\phi(K_i)$ to $K\times D^i$ for all $0\le i\le\sigma-1$. Then, for $i\le\sigma-1$, the homeomorphism
$(\psi \times \text{Id})\circ\phi$ of $K$ to $K\times D^{\sigma-1}\times D=K\times D^\sigma$
maps $K_i$ to $K\times D^i\times\{\infty\}=K\times D^i$.
\end{proof}

\proof[Proof of Proposition \ref{PropHomeo}.] The condition is obviously necessary.

Conversely, set $K_i=C_{\sigma-i}$, and use Lemma \ref{LemOmegaMuK} to get a homeomorphism $\Cond(X)\to\Cond(Y)$ preserving the extended Cantor-Bendixson rank. The hypothesis on scattered parts and Pierce's theorem then allow to get the desired homeomorphism.
\endproof

\begin{lemma}
The spaces $D^m\times [n]$ ($m\ge 0,n\ge 1$), $D^m\times W$ ($m\ge 0$) and $K$ (the Cantor set) are pairwise non-homeomorphic.\label{nonhomeo}
\end{lemma}
\proof
The only perfect space here is $K$. The other uncountable ones are $D^m\times W$, which has Cantor-Bendixson rank $m+1$. The countable space $D^m\times [n]$ has Cantor-Bendixson rank $m+1$ and exactly $n$ points of maximal Cantor-Bendixson rank ($n$).
\endproof

\section{Generalities}\label{gen}

\subsection{Isolated points}

\begin{proposition}
Let $A$ be an abelian group and $S\in\SA$. Then $S$ is isolated in $\SA$ if and only if $S$ is finitely generated and $A/S$ is Artinian.
\end{proposition}

This follows from Lemma 1.3, Proposition 2.1 and Lemma 4.1 in \cite{CGP07}.

\begin{corollary}\label{cor:existisolated}
Let $A$ be an abelian group. Then $\SA$ has isolated points (i.e. is non-perfect) if and only if $A$ is minimax. In this case, isolated points in $\SA$ form exactly one commensurability class, namely the lattices in $A$.\label{i}
\end{corollary}

\begin{corollary}\label{cor:cantor}
If a countable abelian group $A$ is not a minimax group, then $\SA$ is homeomorphic to a Cantor set.
\end{corollary}

These two corollaries settle Proposition \ref{PropNonNA}.

\subsection{Uncountable groups}\label{subs:uncount}

\begin{proof}[Proof of Proposition \ref{uncountable}]
If $A/T_A$ has cardinality $\alpha$, then $A$ contains a copy of $\Z^{(\alpha)}$. Otherwise, $T_A$ has cardinality $\alpha$, and denoting by $A_p$ the $p$-torsion in $A$, the direct sum $\sum_pA_p$ can be written as a direct sum of $\alpha$ cyclic subgroups of prime order.
So in both cases, $A$ contains a subgroup isomorphic to a direct sum of non-trivial (cyclic) subgroups $\bigoplus_{i\in\alpha}S_i$. The mapping $2^A\to\SA$, $J\mapsto\bigoplus_{j\in J}S_j$ is the desired embedding.
\end{proof}

\begin{lemma}Let $\alpha$ be an infinite cardinal. The least cardinal for a basis of open sets in $2^\alpha$ is $\alpha$.\label{topbasis}
\end{lemma}
\proof
The natural basis of topological space $2^\alpha$ has cardinality $\alpha$. Conversely, if a basis has cardinality $\beta$, then it provides a basis of the dense subset $2^{(\alpha)}$. Thus $2^{(\alpha)}$ contains a dense subset $\mathcal{D}$ of cardinality not greater than $\beta$. The union of (finite) supports of all $f\in \mathcal{D}$ must be all of $\alpha$, so $\beta\ge\alpha$.\endproof

\begin{proposition}
Let $\alpha$ be an infinite cardinal, and $A$ an abelian group of cardinal $\alpha$. The least cardinal for a basis of open sets in $\SA$ is $\alpha$.
\end{proposition}
\proof
As we have topological embeddings $2^{\alpha}\subset \SA\subset 2^\alpha$ (the right-hand one being the inclusion $\SA\subset 2^A$, the left-hand one following from Proposition \ref{uncountable}), this follows from Lemma \ref{topbasis}.
\endproof

\section{The weight function on $\SA$}\label{wei}

In this section, all minimax groups are assumed abelian.

\subsection{Critical primes, idle subgroups and parallelism}

\begin{definition}
Let $A$ be a minimax group. Two subgroups $S,S'$ are said to be
\emph{parallel} if they have a common lattice, and
$\ell_p(S)=\ell_p(S')$ for every prime $p$. 
\end{definition}

Clearly, this is an
equivalence relation. Commensurable implies parallel; 
the obstruction to the converse comes from what we call critical primes.

\begin{definition}[Strong Criticality] \label{DefiStrongCritic}
Let $A$ be a minimax group and let $S \le A$.
\begin{itemize}
\item The subgroup $S$ is $p$-critical if $p$ is a critical prime \\(i.e. $\ell_p(A)\ge
2$) and $\ell_p(S)>0$.
\item The subgroup $S$ is strongly $p$-critical if $\ell_p(S)>0$ and \\$\tau_p(A/S)>0$.
\end{itemize}
\end{definition}

\begin{definition}
Let $A$ be a minimax group. A subgroup $S$ of $A$ is \textit{idle} if any subgroup parallel to $S$ is commensurable to $S$.
\end{definition}

\begin{lemma}\label{Lem parallnotstrong}
Let $A$ be a minimax group and let $S\le A$. The following are equivalent:
\begin{itemize}
\item $S$ is idle;
\item $S$ is not strongly $p$-critical for any critical prime $p$.
\end{itemize}
\end{lemma}

\proof Suppose that $S$ is not strongly critical for any critical prime $p$. Let $S'$ be a
subgroup of $A$ parallel to $S$. Replacing $S$ and $S'$ by $S/Z$ and
$S'/Z$ where $Z$ is a common lattice, we can assume
that $S,S'$ are (Artinian) torsion subgroups of $A$.

Let us show that $\Div(S)=\Div(S')$, which clearly proves that $S'$ is
commensurable to $S$. Let $p$ be a prime such that $\tau_p(A)>0$.

If $p$ is non critical (i.e. $\tau_p(A)=1$), then either
$\tau_p(S)=\tau_p(S')=1$ and hence $\Div(S) \cap \Div(S')$ contains
the divisible part $\Div(A_p)$ of the $p$-component $A_p$ of $A$ or
$\tau_p(S)=\tau_p(S')=0$ and neither $\Div(S)$ nor $\Div(S')$ contains
this part.

If $p$ is critical, then there are two cases (remind that $S$ is
not strongly $p$-critical for any critical prime $p$):

Case 1: $\tau_p(S)=0$. Then $\tau_p(S')=0$ and hence both $\Div(S)$
and $\Div(S')$ intersect trivially $\Div(A_p)$.

Case 2: $\tau_p(S)=\tau_p(A)$. Then $\tau_p(S')=\tau_p(A)$ and hence
$\Div(S) \cap \Div(S') \supset \Div(A_p)$. All in all, this shows that
$\Div(S)=\Div(S')$.

Conversely, suppose that $S$ is strongly $p$-critical for some critical prime
$p$. Let $S'$ be the kernel of a homomorphism from $S$ onto $\Cp$.
Since $S/S'$ is torsion, the natural map $A/S' \longrightarrow A/S$
maps $T_{A/S'}$ onto $T_{A/S}$. As $\ell_p(S/S')=1$ we have then
$\tau_p(A/S')=\tau_p(A/S)+1$. As $S$ is strongly $p$-critical, we
have $\tau_p(A/S') \ge 2$. So $A/S'$ contains a subgroup $L/S'$
isomorphic to $\Cp$, which is not equal (and therefore not
commensurable) to $S/S'$. So $L$ is parallel but not commensurable
to $S$.
\endproof

\subsection{The weight function and semi-continuity}

\begin{lemma}\label{suradd}
Let $A$ be a minimax group and $p$ a prime. Then for every $S\in\SA$, we have
$$\tau_p(S)+\tau_p(A/S)\ge\tau_p(A),$$ with equality if $S$ is torsion.\end{lemma}
\proof
It is clearly an equality when the kernel is torsion. Apply this to the exact sequence
$$0\to T_A/(T_A \cap S)\to A/S\to A/(S+T_A)\to 0$$
to get (note that $T_A \cap S=T_S$) $$\tau_p(A/S)=\tau_p(A/(S+T_A))+\tau_p(T_A/T_S).$$
Again, using additivity in the torsion case, $$\tau_p(T_A/T_S)=\tau_p(A)-\tau_p(S).$$
Thus $$\tau_p(A/S)+\tau_p(S)=\tau_p(A)+\tau_p(A/(S+T_A)).\qedhere$$
\endproof

The following lemma is straightforward.

\begin{lemma} \label{rusc0}
Let $A$ be an abelian group. The map $S\mapsto r(S)$ is lower semi-continuous on $\mathcal{S}(A)$. In particular, if $r(A)<\infty$, then the map $S\mapsto r(A/S)$ is upper semi-continuous on $\mathcal{S}(A)$.\label{gammausc0}\qed
\end{lemma}

\begin{lemma}
Let $A$ be an abelian group with finitely many elements of order $p$ (e.g. $A$ is minimax). Then the map $S\mapsto\tau_p(S)$ is upper semi-continuous on $\mathcal{S}(A)$.\label{taupusc0}
\end{lemma}
\begin{proof}
Note that $\tau_p(S)$ makes sense, since the $p$-component of $T_A$, that is, the set of elements whose order is a power of $p$, is Artinian.
Let $(T_A)_p$ be the $p$-component of the torsion in $A$, i.e. the set of element of $p$-prime order. Consider $S \le A$
and let us show that $\tau_p$ is upper semi-continuous at $S$.
There exists a finite subgroup $M$ of $(T_S)_p$ such that $(T_S)_p/M$ is divisible. So taking the quotient by $M$, we can assume that $(T_S)_p=S\cap (T_A)_p$ is divisible. Let $F$ be the set of elements of order $p$ in $A \setminus S$. Then $S$ contains exactly
$p^{\tau_p(S)}-1$ elements of order $p$. Therefore for any $S' \le
A$ with $S'\cap F=\emptyset$, we have $\tau_p(S')\le\tau_p(S)$.
\end{proof}

\begin{lemma} 
Let $A$ be an abelian group with finitely many elements of order $p$, and with $r(A)<\infty$ (e.g. $A$ is minimax). Then the map $S\mapsto\tau_p(A/S)$ is lower semi-continuous on $\mathcal{S}(A)$.\label{taulsc0}
\end{lemma}
\begin{proof}
Note that the assumption on $A$ is inherited by its quotients (we need $r(A)<\infty$ here), so the map considered here makes sense. Let us check that this map is lower semi-continuous at $S_0$. We can suppose that $S_0$ is torsion. Indeed, let $Z$ be a lattice in $S_0$. Then $\mathcal{S}(A/Z)$ can be viewed as an open subset in $\SA$, $S_0$ corresponding to $S_0/Z$.
So assume that $S_0$ is torsion. We can write $\tau_p(A/S)=f(S)+g(S)$, with $f(S)=\tau_p(A/S)+\tau_p(S)$ and $g(S)=-\tau_p(S)$. By Lemma \ref{taupusc0}, $g$ is lower semi-continuous. By Lemma \ref{suradd}, since $S_0$ is torsion, $f$ takes its minimal value $\tau(A)$ at $S_0$, so is lower semi-continuous at $S_0$. 
\end{proof}

\begin{definition}
Let $A$ be a minimax group.
The \emph{weight} of a subgroup $S$ is
$$
w_A(S)=r(A/S)+\ell(S)+\kappa(A/S).
$$
\end{definition}

The following lemma is straightforward.

\begin{lemma}
Let $A$ be a minimax group. The maps $w_A$ is constant on each commensurability class in $\mathcal{S}(A)$.\qed
\end{lemma}

\begin{lemma}\label{mapsusc}
Let $A$ be a minimax group. The maps $w_A$ is upper semi-continuous on $\mathcal{S}(A)$.
\end{lemma}
\begin{proof}
Observe that
$$w_A(S)=r(A/S)+\ell(A)-\tau(A/S).$$
so $w_A$ is upper semi-continuous as consequences of Lemmas \ref{gammausc0} and \ref{taulsc0}.
\end{proof}

\begin{lemma}\label{closeS}
Let $A$ be a minimax group and $S$ a torsion subgroup. If $S'\le A$ is close enough to $S$ and $\Div(S')\le S$, then either $S'=S$ or $w_A(S')<w_A(S)$.
\end{lemma}
\proof
First we can mod out by a finite subgroup of $S$ and suppose that $S$ is divisible. We can write $w_A(S')=\ell(A)+r(A/S')-\sum_p\tau_p(A/S')$ as a sum of upper semi-continuous functions. So if $S'$ be close to $S$ with $w_A(S')=w_A(S)$, we have $r(A/S')=r(A/S)$ by Lemma \ref{rusc0} (hence $S'$ is torsion too) and $\tau_p(A/S')=\tau_p(A/S)$ for all $p$ by Lemma \ref{taulsc0}, i.e. $\ell_p(S')=\ell_p(S)$ for all $p$. As by assumption $\Div(S')\le S$, it follows that $\Div(S')=S$. So $S$ has a direct complement $M$ in $S'$. Let $F$ be the set of elements of prime order in $A\setminus S$. If $S'$ is close enough to $S$, we have $S'\cap F=\emptyset$, hence $M$ cannot contain any element of prime order, so $M=\{0\}$, and $S'=S$.
\endproof

\begin{proposition}\label{prop:neigh}
Let $A$ be a minimax group, and $S\in\SA$. If $S'$ is close enough to $S$, then either
\begin{itemize}\item[(1)] $w_A(S')<w_A(S)$, \item[(2)] or $w_A(S')=w_A(S)$ and $S'$ is parallel and non-commensurable to $S$,\item[(3)] or $S'=S$.
\end{itemize}\end{proposition}
\proof
Let $S\le A$ and let $Z$ be a finitely generated subgroup of $S$. If $S'$ is close enough to $S$, then $Z\le S'$, and as $w_{A/Z}(S'/Z)=w_A(S')$ whenever $S'$ contains $Z$, we can suppose that $S$ is divisible. If $S'$ is close to $S$ and $w_A(S')=w_A(S)$ then $r(A/S')=r(A/S)$, so that $S'$ is torsion, and $\tau_p(A/S')=\tau_p(A/S)$ for all $p$, hence $\ell_p(S)=\ell_p(S')$ for all $p$, i.e. $S$ and $S'$ are parallel (argue as in the beginning of the proof of Lemma \ref{closeS}).
If moreover $S'$ is commensurable to $S$, then $\Div(S')=\Div(S)$. By Lemma \ref{closeS}, we get that if $S'$ is close enough to $S$, then $S'=S$.
\endproof

As $w_A$ is constant on commensurability classes, we get

\begin{corollary}
Let $A$ be a minimax group. Every commensurability class in $\SA$ is discrete.\qed
\end{corollary}

\begin{corollary}\label{cor:sscidle}
Let $A$ be a minimax group and $S\in\SA$. Then the function $w_A$ is strictly upper semi-continuous at $S$ if and only if $S$ is idle.
\end{corollary}
Indeed, the case (2) can occur only if $S$ is non-idle.

If $A$ has no critical primes, every subgroup is idle and we thus get

\begin{corollary}\label{noncritssci}
If the minimax group $A$ has no critical primes, then the map $w_A$ is strictly upper semi-continuous on $\SA$.\qed
\end{corollary}

\subsection{Commensurable convergence}

\begin{definition}
Let $H$ and $S$ be two subgroups of $A$. We say that there is
commensurable convergence from $H$ to $S$ if $S$ belongs to the
topological closure of the commensurability class of $H$.
\end{definition}

\begin{lemma} \label{LemComConv}
Let $A$ be a minimax group and $S \le T_A$ a torsion subgroup. Let $H$ be another
subgroup of $A$. Then the following are equivalent.
\begin{itemize}
\item[(i)] There is commensurable convergence from $H$ to $S$;
\item[(ii)] $\Div(H)$ is contained in $S$.
\end{itemize}
\end{lemma}
\proof Suppose (ii). As a divisible subgroup, $\Div(H)$ has a direct complement $L$ in $H$, and $L$ contains a torsion-free subgroup of finite index $L'$.
Let $(F_n)$ be a non-decreasing sequence of subgroups of $S$
containing $\Div(H)$ as a subgroup of finite index, with union all
of $S$. Let $(V_n)$ be a sequence of subgroups of finite index of $L'$,
with trivial intersection. Set
$$H_n=F_n\oplus V_n\subset S\oplus L'.$$ Then
$H_n$ has finite index in $F_n\oplus L$, which contains $\Div(H)\oplus L=H$ with finite index. So $H_n$ is commensurable to $H$; clearly $(H_n)$
tends to $S$.

Conversely, suppose that $T_H \cap S$ has infinite index in $T_H$.
Then $H$ contains a quasi-cyclic subgroup $P\simeq\Cp$ such that $P\cap S$ is
finite. Then every subgroup of $A$ commensurable to $H$ contains
$P$. Therefore this remains true for every group in the closure of
the commensurability class of $H$, which therefore cannot contain
$S$.
\endproof

As a consequence, we get
\begin{proposition} \label{PropComConv}
Let $A$ be a minimax group and $S$ a subgroup of $A$. Let $H$ be another
subgroup of $A$. Then the following are equivalent.
\begin{itemize}
\item[(i)] There is commensurable convergence from $H$ to $S$;
\item[(ii)] $H$ contains some lattice $Z$ of $S$ and, in $A/Z$, $T_{H/Z}$ is virtually contained in $S/Z$.
\end{itemize}
\end{proposition}
\begin{proof}
The implication (ii)$\Rightarrow$(i) is a direct corollary of Lemma \ref{LemComConv}. Suppose (i). Let $Z_0$ be a lattice of $S$ and let $(H_n)$ be a sequence of subgroups of $A$, commensurable to $H$, converging to $S$. Then $H_n\cap Z_0\to Z_0$. As $Z_0$ is finitely generated, eventually $H_n$ contains $Z_0$. So $H$ contains a finite index subgroup $Z$ of $Z_0$. So working in $A/Z$ and applying Lemma \ref{LemComConv}, we get (ii).
\end{proof}

As an application, we have the two extreme cases

\begin{proposition}\label{extreme_com_class}
Let $A$ be a minimax group.\begin{itemize}\item If $Z$ is a lattice of $A$, then there is commensurable convergence from $Z$ to any subgroup of $A$.\item If $S=T_A$, then $S$ belongs to the closure of any commensurability class in $\SA$.\end{itemize}
\end{proposition}

So, the lattices of a minimax group form a dense commensurability class. As it consists of isolated points it is the unique dense commensurability class. At the opposite, the torsion subgroup is in the unique closed (and finite) commensurability class of $\SA$.

Another consequence of Proposition \ref{PropComConv} is the following, which quite surprisingly is not an obvious consequence of the definition.

\begin{corollary}\label{cctrans}
If there is commensurable convergence $H\to L$ and $L\to S$, then there is commensurable convergence $H\to S$.\qed
\end{corollary}

\subsection{Hereditary properties}

\begin{proposition} \label{PropHeredity0} \noindent
Let $A$ be a minimax group. Set $I=\{S\in\SA|\,\kappa_{\Crit(A)}(A/S)=\ell_{\Crit(A)}(S)=0\}$. Then
the map $w_A$ is strictly hereditary on $I$.
\end{proposition}

The proposition readily follows from the following lemma.

\begin{lemma} \label{LemDIsHered0}
Let $A$ be a minimax group and $S\in\SA$. Suppose that
$$\kappa_{\Crit(A)}(A/S)=0\text{ and }r(A/S)+\ell_{\Ncrit(A)}(S)\ge 1$$ (if $S\in I$, these two assumptions just mean that $w_A(S)\ge 1$). Then there exists
$H\in\SA$ with commensurable convergence $H\to S$, with $w_A(H)=w_A(S)-1$ and
$\ell_{\Crit(A)}(H)=\ell_{\Crit(A)}(S)$.
\end{lemma}
\proof We can assume that $S$ is torsion. Indeed, let $Z$ be a lattice of $S$. We can check that $A/Z$ and $S/Z$ satisfy the same hypotheses as $A$ and $S$, and that $w_{A/Z}(H/Z)=w_A(H)$ for any $H\le A$ containing $Z$.

\begin{itemize}\item If $r(A)\ge 1$, then there exists a torsion-free subgroup $Q \le A$
of rank one such that $\ell_{\Crit(A)}(Q)=0$ and $\tau(A/(S\oplus Q))=\tau(A/S)$. Indeed, take a direct complement $L$ of $\Div(A)$ in $A$ and a torsion-free finite index subgroup $L'$ of $L$, a cyclic subgroup $C \le L'$, and the inverse image $Q \le L'$
of the torsion of $L'/C$. As $0=\kappa_{\Crit(A)}(A/S)=\ell_{\Crit(A)}(L)=0$, we have $\ell_{\Crit(A)}(Q)=0$. Now $\tau(A/(S\oplus Q))=\tau(T_{A/S}\oplus L/Q)=\tau(A/S)$.

Setting $H=S\oplus Q$, we
have $\ell_{\Crit(A)}(H)=\ell_{\Crit(A)}(S)$ and $$\begin{array}{ccc}
w_A(H)&=&r(A/H)+\ell(A)-\tau(A/H)\\
&=& r(A/S)-1+\ell(A)-\tau(A/S)\\ & = & w_A(S)-1.
\end{array}$$

 Now $S$ is the limit of the subgroups $H_k:=k!Q \oplus S \le H$ ($k \in \N$), which
have finite index in $H$. Indeed, observe that 
\begin{itemize}
\item $nQ$ is close to $\{0\}$ if $n$ has some large prime divisor $p$ such that no elements of $Q$ is divisible by every power of $p$;
\item $nQ$ has finite index in $Q$ for any $n\ge1$, because $Q/nQ$ is Artinian with finite exponent.
\end{itemize}

\item If $r(A)=0$, then $A$ is Artinian, and by assumption this forces $\ell_{\Ncrit(A)}(S)\ge 1$. So $S$
has a direct summand $P$ isomorphic to $\Cp$ for some non-critical
prime $p$. Let $H$ be a direct complement of $P$ in $S$. We have
$w_A(H)=w_A(S)-1$ and $\ell_{\Crit(A)}(H)=\ell_{\Crit(A)}(S)$.

Denote by $H_k$ the direct sum of $H$ with the subgroup of $P$ of
order $p^k$. Then $H_k$ is commensurable to $H$ and $H_k$ tends to
$S$ as $k$ goes to infinity.\qedhere
\end{itemize}\endproof

\subsection{Conclusion in the scattered case}

\begin{theorem}\label{thm:scat}
Let $A$ be a minimax group with no critical primes. Then $\SA$ is scattered, and the Cantor-Bendixson rank of $S\in\SA$ is given by $w_A(S)$. The Cantor-Bendixson rank of $\SA$ is $h(A)+1$, and $\SA$ is homeomorphic to $D^{h(A)}\times [n(A)]$, where $n(A)$ is the number of subgroups commensurable to $T_A$.
\end{theorem}
\proof Suppose that $A$ has no critical primes. Then the subset $I$ of Proposition \ref{PropHeredity0} coincides with all of $\SA$, so that $w_A$ is strictly hereditary on $\SA$. Moreover, $w_A$ is strictly upper semi-continuous on $\SA$ by Corollary \ref{noncritssci}. So we can apply Lemma \ref{carcb} (with $C=\emptyset$) to obtain that $\SA$ is scattered and that the Cantor-Bendixson rank of an element $S\in\SA$ is $w_A(S)$.

Writing $w_A(S)=h(A)-(r(S)+\tau(A/S))$, we see that the maximal value of $w_A$ is given by $h(A)$, and is attained exactly for subgroups commensurable to $T_A$.
\endproof


\section{The leveled weight and condensation on $\SA$}\label{lwei}

Again, in this section, all Artinian and minimax groups are assumed abelian.

\subsection{The invariant $\gamma$}

Let $V$ be a set of primes, and $\Z_V$ be the ring of rationals whose denominator has no divisor in $V$.

Let $A$ be an abelian group. Then $\Hom(A,\Z_V)$ is a torsion-free $\Z_V$-module. Its rank, i.e. the dimension of the $\Q$-vector space $\Hom(A,\Z_V)\otimes_{\Z_V}\Q$, is denoted by $\gamma_V(A)$. Note that $\gamma_\emptyset(A)=r(A)$.

\begin{lemma}
If $A$ has finite rank $r(A)$, then $\Hom(A,\Z_V)$ is a finitely generated $\Z_V$-module, free of rank $\gamma_V(A)$, and $\gamma_V(A)\le r(A)$.
\end{lemma}
\begin{proof}
If $n=r(A)$ and $(e_1,\dots,e_n)$ is a maximal $\Z$-free family in $A$, then the mapping $f\mapsto (f(e_1),\dots,f(e_n))$ embeds $\Hom(A,\Z_V)$ as a submodule of the free module of rank $n$.
In particular, as $\Z_V$ is principal, $\Hom(A,\Z_V)$ is a free $\Z_V$-module of rank $\gamma_V(A)$ and $\gamma_V(A)\le r(A)$.
\end{proof}

Define $a_V(A)=r(A)-\gamma_V(A)$. Besides, define $U_V(A)$ as the intersection of all kernels of homomorphisms $A\to\Z_V$; this is a characteristic subgroup of $A$.

\begin{lemma} \label{LemAQuot}
Let $A$ be an abelian group of finite rank and let $S \le A$. We have then:
 $$\gamma_V(A)\ge\gamma_V(A/S) \ge \gamma_V(A)-\gamma_V(S);$$
$$ a_V(A)-r(S)\le a_V(A/S)\le a_V(A)-a_V(S).$$
In particular, we have $\gamma_V(A/S)=\gamma_V(A)$ and $a_V(A/S)=a_V(A)$ if $S$ is torsion.
\end{lemma}
\begin{proof}
From the exact sequence $0\to S\to A\to A/S\to 0$, we get the exact sequence of $\Z_V$-modules $$0\to\Hom(A/S,\Z_V)\to\Hom(A,\Z_V)\to \Hom(S,\Z_V),$$ so that $\Hom(A,\Z_V)$ lies in an extension of $\Hom(A/S,\Z_V)$ by some submodule of $\Hom(S,\Z_V)$.
This gives the first inequality, and the second one is equivalent to it.
\end{proof}

\begin{lemma}
Let $A$ be an abelian group of finite rank. We have $a_V(A)=r(U_V(A))$ and $\gamma_V(A)=r(A/U_V(A))$.
\end{lemma}
\begin{proof}
The two statements are obviously equivalent; let us prove that $\gamma_V(A)=r(A/U_V(A))$.

Set $B=A/U_V(A)$. If $i:\Z^{r(B)}\to B$ is an embedding as a lattice, then, as $\Hom(B/\text{Im}(i),\Z_V)=\{0\}$, it induces an embedding of $\Hom(B,\Z_V)$ into\\ $\Hom(\Z^{r(B)},\Z_V)={\Z_V}^{r(B)}$, which is a $\Z_V$-module homomorphism. So $\gamma_V(A)\le r(A/U_V(A))$.

Conversely, if $(f_1,\dots,f_m)$ is a maximal $\Z$-free family in $\Hom(A,\Z_V)$, with $m=\gamma_V(A)$, then $A/\bigcap\Ker(f_i)$ embeds into ${\Z_V}^m$. But $\bigcap\Ker(f_i)$ is reduced to $U_V(A)$, so $r(A/U_V(A))\le\gamma_V(A)$.
\end{proof}

\begin{corollary}
Let $A$ be a minimax group. We have $a_V(A)=0$ if and only if $\kappa_p(A)=0$ for every $p\in V$.\label{azero}
\end{corollary}
\proof If $a_V(A)=0$, then $U_V(A)$ is torsion and hence coincides with $T_A$. Therefore $A/T_A$ embeds into ${\Z_V}^{r(A)}$. As a result $\kappa_p(A)=\ell_p(A/T_A)=0$ for every $p\in V$. Conversely, suppose that $\kappa_p(A)=0$ for every $p\in V$. We have an embedding of $A/T_A$ into $\Q^{r(A)}$. The assumption implies, using Lemma \ref{multiple} below, that the image is contained in a multiple of ${\Z_V}^r(A)$. So $U_V(A)$ is torsion, i.e. $a_V(A)=0$.
\endproof

\begin{lemma}\label{multiple}
Let $A$ be a minimax subgroup of $\Q$ and $n\ge 1$. Suppose that $\ell_p(A)=0$ for every prime $p$ not dividing $n$. Then $A\le\lambda\Z[1/n]$ for some $\lambda\in\Q^*$. 
\end{lemma}
\proof Let $B=\bigoplus B_p$ be the image of $A$ in $\Q/\Z=\bigoplus_p\Cp$. As $A$ is a minimax group, $B_p=0$ for all but finitely many $p$'s, and $B_p$ is finite when $p$ does not divide $n$. Hence $B$ is virtually contained in $\bigoplus_{p|n}\Cp=\Z[1/n]/\Z$. So $A$ is virtually contained in $\Z[1/n]$. As $A$ is locally cyclic, $A$ is generated by $A\cap\Z[1/n]$ and some rational $u/v$. Then $A\le v^{-1}\Z[1/n]$.
\endproof

\begin{lemma}
Let $A$ be an abelian group with $r(A)<\infty$. The maps $S\mapsto a_V(A/S)$, $S\mapsto \gamma_V(A/S)$ are upper semi-continuous on $\mathcal{S}(A)$.\label{gammausc}
\end{lemma}
\begin{proof}
Take $S_0\in\mathcal{S}(A)$. By the same argument as in the proof of Lemma \ref{taulsc0}, we can suppose that $S_0$ is torsion.

By the first (resp. second) inequality in Lemma \ref{LemAQuot}, $S\mapsto\gamma_V(A/S)$ (resp. $a_V(A/S)$) is maximal at $S_0$, so is upper semi-continuous at $S_0$.
\end{proof}

\subsection{More maps on $\SA$}

Let $A$ be a minimax group, and recall that $\Crit(A)$, respectively $\Ncrit(A)$, denotes the set of critical primes of $A$ (resp. the set of non-critical primes of $A$).
We define three maps $\mathcal{S}(A) \rightarrow \mathbf{N}$
\begin{enumerate}
\item The \emph{level} of a subgroup $S$ is
$$
\lambda_A(S)=a_{\Crit(A)}(A/S)+\ell_{\Crit(A)}(S)+\kappa_{\Crit(A)}(A/S);
$$
\item The \emph{leveled weight} is
$$d_A(S)=\gamma_{\Crit(A)}(A/S)+\ell_{\Ncrit(A)}(S)+\kappa_{\Ncrit(A)}(A/S).$$
\end{enumerate}
Note that $w_A=\lambda_A+d_A$.

\begin{lemma}
Let $A$ be a minimax group. The maps $\lambda_A$ and $d_A$ are constant on each commensurability class in $\mathcal{S}(A)$.
\end{lemma}
\begin{proof}This is clear for the maps $S\mapsto\kappa_p(A/S)$, $S\mapsto\ell_p(S)$. For $S\mapsto\gamma_V(A/S)$, let $S,S'$ be commensurable subgroups of $A$. We can deduce from \ref{LemAQuot} that $\gamma_V(A/S)=\gamma_V(A/(S\cap S'))=\gamma_V(A/S')$. Finally $a_V(A/S)=r(A/S)-\gamma_V(A/S)$ is settled.
\end{proof}

\begin{lemma}\label{mapsusc2}
Let $A$ be a minimax group. The maps $\lambda_A$ and $d_A$ are upper semi-continuous on $\mathcal{S}(A)$.
\end{lemma}
\begin{proof}
Observe that
$$\lambda_A(S)=a_{\Crit(A)}(A/S)+\ell_{\Crit(A)}(A)-\tau_{\Crit(A)}(A/S);$$
$$d_A(S)=\gamma_{\Crit(A)}(A/S)+\ell_{\Ncrit(A)}(A)-\tau_{\Ncrit(A)}(A/S);$$
so they are upper semi-continuous as consequences of Lemmas \ref{gammausc0}, \ref{taulsc0}, and \ref{gammausc}.
\end{proof}

\begin{proposition}\label{LemWscsX0}
In restriction to the open subset $I=\{\lambda_A=0\}$ of $\mathcal{S}(A)
$, the map $w_A\;(=d_A)$ is strictly upper semi-continuous. 
\end{proposition}
\proof
If $\lambda_A(S)=0$, $\ell_p(S)=0$ for every critical prime $p$, then $S$ is not strongly $p$-critical for any $p$. By Lemma \ref{Lem parallnotstrong}, $S$ is idle.
\endproof

\subsection{Hereditary properties (second part)}

If $A$ is a minimax group, in view of Corollary \ref{azero}, the set $I$ of Proposition \ref{PropHeredity0} coincides with $\{S\in\SA|\,\lambda_A(S)=0\}$. Accordingly Proposition \ref{PropHeredity0} states that, on $I$, the function $w_A\,(=d_A)$ is strictly hereditary.


\begin{lemma}
Let $A$ be a minimax group and $S\in\SA$ with $\lambda_A(S)\le 1$. Then $a_{\Crit(A)}(A/S)=\kappa_{\Crit(A)}(A/S)=0$.\label{zeroboth}
\end{lemma}
\proof
We have $\lambda_A(S)=a_{\Crit(A)}(A/S)+\ell_{\Crit(A)}(S)+\kappa_{\Crit(A)}(A/S) \le 1$. By Lemma \ref{azero},
$\kappa_{\Crit(A)}(A/S)=0$ if and only if $a_{\Crit(A)}(A/S)=0$, so they are both zero.
\endproof
Using this, Lemma \ref{LemDIsHered0} can be restated in the following form

\begin{lemma} \label{LemDIsHered}
Let $A$ be a minimax group and $S\le A$. Suppose that $\lambda_A(S)\le 1$ and $d_A(S)\ge 1$. Then there exists
$H$ with commensurable convergence $H\to S$, with $d_A(H)=d_A(S)-1$ and
$\lambda_A(H)=\lambda_A(S)$.
\end{lemma}

\proof
By Lemma \ref{LemDIsHered0}, there exists $H$ with commensurable convergence $H\to S$, with $\ell_{\Crit(A)}(H)=\ell_{\Crit(A)}(S)$ and $w_A(H)=w_A(S)-1$. By upper semi-continuity of $\lambda_A$, we have $\lambda_A(H)\le\lambda_A(S)$. By Lemma \ref{zeroboth}, $\lambda_A(S)=\ell_{\Crit(A)}(S)=\ell_{\Crit(A)}(H)\le\lambda_A(H)$. So $\lambda_A(H)=\lambda_A(S)$, hence $d_A(H)=d_A(S)-1$.
\endproof

\begin{lemma} \label{LemLambdaIsHered}
Suppose that $\lambda_A(S)\ge 1$. Then there exists $H$ with
commensurable convergence $H\to S$, with $\lambda_A(H)=\lambda_A(S)-1$
and $d_A(H)=d_A(S)$.
\end{lemma}
\proof
We can suppose that $S$ is torsion (argue as in the beginning of the proof of Lemma \ref{LemDIsHered0}). We easily check from the definition that $\kappa_p(A/S)=\kappa_p(A)$ for any prime $p$. By
Lemma \ref{LemAQuot}, we also have $a_{\Crit(A)}(A/S)=a_{\Crit(A)}(A)$. Therefore
$\lambda_A(S)=a_{\Crit(A)}(A)+\ell_{\Crit(A)}(S)+\kappa_{\Crit(A)}(A)$.

Suppose that $a_{\Crit(A)}(A)\ge 1$, i.e. that $U_{\Crit(A)}$ is not torsion. Write $A=T_A\oplus Q$ with $Q$
torsion-free. We have $U_{\Crit(A)}(A)=U_{\Crit(A)}(T_A)\oplus U_{\Crit(A)}(Q)$, so $U_{\Crit(A)}(Q)$ is not torsion. 
Let $Z$ be an infinite cyclic subgroup of $U_{\Crit(A)}(Q)$, and
let $L$ be the inverse image in $Q$ of the torsion subgroup of $Q/Z$. Finally
set $H=S\oplus L$. We have (note that $\gamma_{\Crit(A)}(A/S)=\gamma_{\Crit(A)}(A)$ by Lemma \ref{LemAQuot}):
$$d_A(H)-d_A(S)=\gamma_{\Crit(A)}(A/H)-\gamma_{\Crit(A)}(A)-\tau_{\Ncrit(A)}(A/H)+\tau_{\Ncrit(A)}(A/S).$$

We claim that:
\begin{itemize}
\item[$(1)$] $\tau_p(A/H)=\tau_p(A/S)$ for every prime $p$;
\item[$(2)$] $\gamma_{\Crit(A)}(A/H)=\gamma_{\Crit(A)}(A)$ and $a_{\Crit(A)}(A/H)=a_{\Crit(A)}(A)-1$.
\end{itemize}

\emph{Proof of Claim $(1)$.} As $Q/L$ is torsion-free the groups $T_{A/H}$ and $T_{A/S}$ are both isomorphic to $T_A/S$. Hence $\tau_p(A/H)=\tau_p(A/S)$ for every $p$.

\emph{Proof of Claim $(2)$} By Lemma \ref{LemAQuot}, we have
$\gamma_{\Crit(A)}(A/H)=\gamma_{\Crit(A)}(A/Z)$ and $a_{\Crit(A)}(A/H)=a_{\Crit(A)}(A/Z)$. 
Since $Z \le U_{\Crit(A)}(A)$, we have $\Hom(A/Z,Z_{\Crit(A)})=\Hom(A,Z_{\Crit(A)})$ and hence $\gamma_{\Crit(A)}(A/Z)=\gamma_{\Crit(A)}(A)$, or equivalently $a_{\Crit(A)}(A/Z)=a_{\Crit(A)}(A)-r(Z)=a_{\Crit(A)}(A)-1$, which completes the proof
of claim $(2)$.

We deduce from the previous claims that $d_A(H)=d_A(S)$ and\\
$\lambda_A(H)=\lambda_A(S)-1$. Moreover, $S$ is a strict limit of
subgroups of finite index in $H$.

Finally, suppose that $a_{\Crit(A)}(A)=0$ (which implies $\kappa_{\Crit(A)}(A)=0$ and
$\tau_{\Crit(A)}(S)\ge 1$). Write $S=P\oplus L$, where $P$ is isomorphic to
$\Cp$ for some critical prime $p$. Then $d_A(S)=d_A(H)$ and
$\lambda_A(H)=\lambda_A(S)-1$. Moreover, $S$ is a strict limit of
subgroups containing $H$ as a subgroup of finite index.
\endproof

\begin{lemma} \label{lambda1}
Let $A$ be a minimax group and $S$ a subgroup, and $m\le\lambda_A(S)$. Then there exists $H$ with
commensurable convergence $H\to S$, with $\lambda_A(H)=m$
and $d_A(H)=d_A(S)$.
\end{lemma}
\proof This is a straightforward induction on $\lambda_A(S)-m$, based on Lemma \ref{LemLambdaIsHered} and making use of transitivity of commensurable convergence (Corollary \ref{cctrans}).
\endproof

\begin{proposition} \label{PropHereditybis} \noindent
Let $A$ be a minimax group. Again, set $I=\{S\in\SA|\lambda_A(S)=0\}$. Then the map $d_A$ is $I$-hereditary on $\SA$.
\end{proposition}

\begin{proof}
First, the statement presupposes that $I$ is dense in $\SA$, which is a consequence of Proposition \ref{extreme_com_class} and Corollary \ref{cor:existisolated} (but also follows from the argument below).

Pick $S \le A$. Let $V$ be an open neighbourhood of $S$.
Using Lemma \ref{lambda1} with $m=0$, we obtain that
there exists $H \in I\cap V$ such that $d_A(H)=d_A(S)$. Hence $d_A$ is $I$-hereditary at $S$.
\end{proof}

\begin{lemma} \label{lambda1more}
Let $A$ be a minimax group and $S$ a subgroup of $A$ with $\lambda_A(S)\ge 1$, and $1\le n\le d_A(S)$. Then there exists $H$ with
commensurable convergence $H\to S$, with $\lambda_A(H)=1$
and $d_A(H)=n$.
\end{lemma}
\proof By Lemma \ref{lambda1}, there exists $H\in\SA$ with commensurable convergence $H\to S$, and $\lambda_A(H)=1$, $d_A(H)=d_A(S)$. Applying several times Lemma \ref{LemDIsHered} (using Corollary \ref{cctrans}), we find $H$ with commensurable convergence $H\to S$, $d_A(H)=n$ and $\lambda_A(H)=1$. Again with Corollary \ref{cctrans}, we have commensurable convergence $H\to S$.
\endproof

The following lemma gives somehow the ``smallest'' examples of
minimax groups whose space of subgroups is not scattered
(equivalently uncountable).
\begin{lemma} \label{LemCp2}
Let $A=(\Cp)^2$. The set $\mathcal{K}$ of subgroups of $A$ isomorphic to $\Cp$ is
homeomorphic to the Cantor set. In particular
$\mathcal{S}(A)$ is uncountable.
\end{lemma}
\proof By direct computation $\Aut(A)=\GL_2(\Z_p)$, the group of $2
\times 2$ matrices with coefficients in the $p$-adics, acting on $(\Cp)^2$ through its identification with $(\Q_p/\Z_p)^2$. The action on
the set of subgroups is easily checked to be continuous. The
stabilizer of the ``line" $\Cp\oplus\{0\}$ is the set of upper
triangular matrices $T_2(\Z_p)$. The quotient can be identified on
the one hand with the projective line $\mathbf{P}^1(\Q_p)$, which is
known to be homeomorphic to the Cantor space, and on the other hand
with the orbit of $\Cp\oplus\{0\}$.

Now let $P\in \mathcal{K}$. Being
divisible, it has a direct complement $Q$ in $A$. Necessarily, $Q$ is
isomorphic to $\Cp$. Therefore there exists an automorphism of $A$
mapping $\Cp\oplus\{0\}$ to $P$. Thus $\mathcal{K}$ coincides with the 
orbit of $\Cp \oplus \{0\}$, which completes the proof.
\endproof

\begin{lemma} \label{LemCritCantor}
Suppose that $S$ is strongly $p$-critical. Then there exists a
Cantor set $\mathcal{K}\subset\SA$ such that $S\in\mathcal{K}$, and
every $K\in\mathcal{K}-\{S\}$ is parallel and non-commensurable to
$S$.\end{lemma}
\proof Let $W$ be the kernel of a map of $S$ onto $\Cp$. Working
inside $A/W$, we can suppose that $W=0$, i.e. $S$ is isomorphic to
$\Cp$. As $S$ is strongly $p$-critical, there exists another
subgroup $P$ which is isomorphic to $\Cp$ and such that $P\cap S=0$.
Then the set of subgroups of $P \oplus S$ that are isomorphic to
$\Cp$ is homeomorphic to a Cantor set, by Lemma \ref{LemCp2}. They
are all parallel to $S$ and pairwise non-commensurable.
\endproof

\begin{lemma} \label{LemLambda1}
Suppose that $\lambda_A(S)=1$. Then $S$ belongs to a Cantor set whose
points are subgroups $H$ with $d_A(S)=d_A(H)$ and $\lambda_A(H)=1$.
\end{lemma}
\proof By Lemma \ref{zeroboth}, we have $a_{\Crit(A)}(A/S)=\kappa_{\Crit(A)}(A/S)=0$.
So $\ell_{\Crit(A)}(S)=1$ and $\tau_{\Crit(A)}(A/S)=\ell_{\Crit(A)}(A/S) \ge 1$. Let $p$ be the critical prime such that
$\ell_p(S)=1$. Then $S$ is strongly $p$-critical. Thus we conclude
by Lemma \ref{LemCritCantor}.
\endproof

Let $A$ be a minimax group. For $n\ge 0$, define the subset $C_n(A)=\{S|\lambda_A(S)\ge 1,d_A(S)\ge n\}$. By upper semi-continuity (Lemma \ref{mapsusc2}), it is closed.

\begin{proposition}\label{cperfect}
Let $A$ be a minimax group and $n\ge 0$. The subset $C_n(A)\subset\SA$ is perfect.
\end{proposition}
\proof
By Lemma \ref{lambda1} with $m=1$, its subset $C_n^1(A)=\{S\in C_n|\lambda_A(S)=1\}$ is dense in $C_n(A)$. By Lemma \ref{LemLambda1}, $C_n^1(A)$ is perfect, so $C_n(A)$ is perfect as well. 
\endproof

\begin{proposition}\label{emptint}
Let $A$ be a minimax group and $n\ge 0$. Then $C_{n+1}(A)$ has empty interior in $C_n(A)$.
\end{proposition}
\proof Let $S$ belong to $C_{n+1}(A)$. By Lemma \ref{lambda1more} there exists $H\in\SA$ with commensurable convergence $H\to S$, $\lambda_A(H)=1$ and $d_A(H)=n$. So $H$ and its commensurable subgroups belong to $C_n(A)$ but not to $C_{n+1}(A)$, and therefore $S$ is not in the interior of $C_{n+1}(A)$.
\endproof

\subsection{Maximal values}

\begin{lemma}
Let $A$ be a minimax group. On the set $I=\{S|\lambda_A(S)=0\}$, the maximal value of $w_A(=d_A)$ is $\gamma_{\Crit(A)}(A)+\ell_{\Ncrit(A)}(A)$.\label{maxsigma}
\end{lemma}
\proof
We have $$d_A(S)=\gamma_{\Crit(A)}(A/S)+\ell_{\Ncrit(A)}(A)-\tau_{\Ncrit(A)}(A/S)\le \gamma_{\Crit(A)}(A)+\ell_{\Ncrit(A)}(A)$$
and
$$\lambda_A(S)=a_{\Crit(A)}(A/S)+\ell_{\Crit(A)}(A)-\tau_{\Crit(A)}(A/S)\le a_{\Crit(A)}(A)+\ell_{\Crit(A)}(A).$$
Both inequalities are sharp, as they become equalities when $S=T_A$.

Moreover by Lemma \ref{lambda1}, we get the existence of a subgroup $S$ of $A$ with $\lambda_A(S)=0$ and $d_A(S)=\gamma_{\Crit(A)}(A)+\ell_{\Ncrit(A)}(A)$.
\endproof

\subsection{Conclusion in the non-scattered case}

Define $$\sigma(A)=\gamma_{\Crit(A)}(A)+\ell_{\Ncrit(A)}(A)=h(A)-(a_{\Crit(A)}(A)+\ell_{\Crit(A)}(A)).$$

\begin{theorem}\label{thm:nonscat}
Let $A$ be a minimax group. If $S\in\SA$, then
\begin{itemize}
\item $S$ belongs to the scattered part of $\SA$ if and only if $\lambda_A(S)=0$;
\item the extended Cantor-Bendixson rank of $S$ in $\SA$ is $d_A(S)$;
\item the Cantor-Bendixson rank of $\SA$ is $\sigma(A)+1$.
\end{itemize}
If moreover $A$ has at least one critical prime, then $\SA$ is homeomorphic to $D^{\sigma(A)}\times W$.
\end{theorem}
\proof
If $I=\{\lambda_A=0\}$ and $C=\{\lambda_A\ge 1\}$, then by Proposition \ref{cperfect} (for $n=0$), Lemma \ref{mapsusc2}, Propositions \ref{PropHereditybis}, \ref{LemWscsX0} and \ref{PropHeredity0}, $I$, $C$ and $d_A$ satisfy the hypotheses of Lemma \ref{carcb}. This settles the first two assertions, and the third follows from Lemma \ref{maxsigma}.

For the last statement, we appeal to Proposition \ref{PropHomeo}, with $C_i=C_i(A)$ as above. By Proposition \ref{cperfect}, $C_i$ is perfect. As $d_A(T_A)=\sigma(A)$ and $\lambda_A(T_A)=a_{\Crit(A)}(A)+\ell_{\Crit(A)}(A) \ge 1$  (see the proof of Lemma \ref{maxsigma}), we have $T_A\in C_{\sigma(A)}$. Thus $C_i$ is homeomorphic to the Cantor set for all $i\le\sigma(A)$. Finally $C_{i+1}$ has empty interior in $C_i$ for all $i$, by Proposition \ref{emptint}. 
\endproof

\subsection{Example: Artinian groups}

For an Artinian group $A$, the invariant $\gamma_V$ vanishes for every set of primes $V$. Thus, by definition

$$\sigma(A)=\ell_{\Ncrit(A)}(A),$$
and, for $S\in\SA$,
$$w_A(S)=\ell(S);\quad \lambda_A(S)=\ell_{\Crit(A)}(S);\quad d_A(S)=\ell_{\Ncrit(A)}(S).$$

The properties of these maps established above are even easier to obtain in this particular case. Indeed, if $A$ is decomposed as a direct sum of its $p$-components: $A=\bigoplus A_p$, then $\SA=\bigoplus\mathcal{S}(A_p)$.
Then, given that $W\times [n]$ for $n\ge 1$ and $W\times W$ are homeomorphic to $W$, we are reduced to study $\SA$ when $A$ is an Artinian $p$-group.

If $\ell(A)=0$ then $A$ is finite and $\SA$ is homeomorphic to $[n]$ for $n=n(A)$, the number of subgroups of $A$.

If $\ell(A)=1$ then finite subgroups of $A$ are isolated, and form a dense subset, while there are only finitely many infinite subgroups, namely finite index subgroups of $A$. If there are $n$ many such subgroups, then $\SA$ is homeomorphic to $D\times [n]$.

If $\ell(A)=\ell_p(A)\ge 2$, again finite subgroups form a dense subset consisting of isolated points. If $C$ is the set of infinite subgroups, it is then closed. Now $C$ contains a dense subset, namely the set $L_1$ of subgroups $S$ with $\ell(S)=1$. The set $L_1$ is perfect, by Lemma \ref{LemLambda1}. So $C$ is perfect. Accordingly, $\SA$ is homeomorphic to $W$.

%
%
%
%
%
%
%
%
%
%
%
%
%
%
%
%
%
%
%
%
%
%
%
%
%
%
%
%
%
%
%
%
%
%
%

\bibliographystyle{alpha}

\baselineskip=16pt




\bigskip
\footnotesize

\end{document}